\documentclass[a4paper,12pt,final]{amsart}
\usepackage{times,a4wide}

\newcommand{\sinc}{\operatorname{sinc}}
\newcommand{\Dn}{{\mathbb D}^n}
\newcommand{\D}{\mathbb{D}}
\newcommand{\C}{\mathbb{C}}

\newcommand{\lone}{\vert\hspace{-.5mm}\vert\hspace{-.5mm}\vert}
\newcommand{\rone}{\vert\hspace{-.5mm}\vert\hspace{-.5mm}\vert_1}

\newtheorem{theorem}{Theorem}
\newtheorem{lemma}{Lemma}

\title{The Sidon constant for homogeneous polynomials}

\author{Joaquim Ortega-Cerd\`a}
\address{Dept.\ Matem\`atica Aplicada i An\`alisi,
Universitat  de Barcelona, Gran Via 585, 08071 Barce- lona, Spain}
\email{jortega@ub.edu}

\author{Myriam Ouna\"{\i}es}
\address{Institut de Recherche Math\'ematique Avanc\'ee,
Universit\'e de Strasbourg, 7 Rue Ren\'e Des\-car\-tes,
67084 Strasbourg CEDEX, France}
\email{ounaies@math.u-strasbg.fr}

\author{Kristian Seip}
\address{Department of Mathematical Sciences,
Norwegian University of Science and  Technology, NO--7491
Trondheim, Norway} \email{seip@math.ntnu.no} 


\date{\today}

\thanks{The first author is supported by the project MTM2008-05561-C02-01
and the third author is supported by the Research Council of
Norway grant 185359/V30.}

\begin{document}
\begin{abstract}
The Sidon constant for the index set of nonzero $m$-homogeneous
polynomials $P$ in $n$ complex variables is the supremum of the
ratio between the $\ell^1$ norm of the coefficients of $P$ and the
$H^\infty(\D^n)$ norm of $P$. We present an estimate which gives the
right order of magnitude for this constant, modulo a factor
depending exponentially on $m$. We use this result to show that the
Bohr radius for the polydisc $\D^n$ is bounded from below by a
constant times $\sqrt{(\log n)/n}$.
\end{abstract}

\maketitle

\section{Introduction}

This note presents an estimate on the Sidon constant $S(m,n)$ for
the index set of homogeneous polynomials of degree $m$ in $n$
complex variables. The result is optimal in the sense that the
exact value of $S(m,n)$ is determined up to a factor depending
exponentially on $m$. We will use this estimate to find the
precise asymptotic behavior of the $n$-dimensional Bohr radius,
which was introduced and studied by H.~Boas and D.~Khavinson
\cite{BoasKhav97}.

The Sidon constant $S(m,n)$ for the index set $\{\alpha=(\alpha_1,
\alpha_2,\ldots,\alpha_n): \ |\alpha|=m\}$ is defined in the
following way. Let
\[ P(z)=\sum_{|\alpha|=m} c_\alpha z^\alpha \]
be a homogeneous polynomial of degree $m$ in $n$ complex
variables. We let $\D^n$ denote the unit polydisc in $\C^n$ and
set
\[\| P \|_\infty=\sup_{z\in \D^n} |P(z)| \qquad \text{and} \qquad
 \lone P \rone=\sum_{|\alpha|=m}
|c_\alpha|;\] then $S(m,n)$ is the smallest constant $C$ such that
the inequality $ \lone P\rone\le C \| P \|_\infty$ holds for every
$P$. It is plain that $S(1,n)=1$ for all $n$, and this case is
therefore excluded from our discussion. Our main result is the
following estimate.

\begin{theorem} \label{sidon}
There exists an absolute constant $C$ such that the Sidon constant
$S(m,n)$ satisfies
\begin{equation} \label{sharp} S(m,n) \le C^m
\sqrt{\frac{n^{m-1}}{(m-1)!}}
\end{equation}
when $ n> m^2>1$.
\end{theorem}

The Sidon constant $S(m,n)$ is effectively the same as the
unconditional basis constant for the monomials of degree $m$ in
$H^\infty(\D^n)$; the latter is larger than $S(m,n)$ by a factor not
exceeding $2$. This and similar unconditional basis constants were
studied in \cite{DDGM01}, where it was established that $S(m,n)$ is
bounded from above and below by $n^{(m-1)/2}$ times constants
depending only on $m$. The more precise estimate
\begin{equation} \label{old} S(m,n) \le C^m n^{\frac{m-1}{2}},
\end{equation} with $C$ an absolute constant, can be extracted from
\cite{DefFrer06}. By H\"{o}lder's inequality, \eqref{old} also
follows from an interesting inequality of H.~Queff\'{e}lec
\cite{Que}, which says that the $\ell^{2m/(m+1)}$ norm of the
coefficients of a homogeneous polynomial $P$ of degree $m$ is
bounded by $\|P\|_\infty$ times a certain precise constant
depending only on $m$. A more direct deduction of \eqref{old} is
implicit in the work of F.~Bohnenblust and E.~Hille \cite{BH31};
this approach has inspired our proof of \eqref{sharp}.

Note that we also have the following trivial estimate:
\begin{equation} \label{trivial}
S(m,n) \le \sqrt{\binom{n+m-1}{m}},
\end{equation}
which is a consequence of the Cauchy--Schwarz inequality along with the fact
that the number of different monomials of degree $m$ in $n$ variables is
$\binom{n+m-1}{m}$. Comparing \eqref{sharp} and \eqref{trivial}, we see that our
estimate gives a nontrivial result only in the range $\log n >m$. Using the
Salem--Zygmund inequality for random trigonometric polynomials (see \cite[p.
68]{Kahane}), one may check that the estimates \eqref{trivial} and \eqref{sharp}
together give the right value for $S(m,n)$, up to a factor less than $c^m$ with
$c<1$ an absolute constant.

Our application of Theorem~\ref{sidon} to the asymptotic behavior
of the Bohr radius for the polydisc will further illuminate the
significance of \eqref{sharp}. Following \cite{BoasKhav97}, we now
let $K_n$ be the $n$-dimensional Bohr radius, i.e., the largest
positive number $r$ such that all polynomials $\sum_\alpha
c_\alpha z^\alpha$ satisfy
\[
\sup_{z\in r \Dn}\sum_\alpha |c_\alpha z^\alpha| \le \sup_{z\in
\Dn}\Bigl|\sum_\alpha c_\alpha z^\alpha\Bigr|.
\]
The classical Bohr radius $K_1$ was studied and estimated by
H.~Bohr \cite{Bohr14} himself, and it was shown independently by
M.~Riesz, I.~Schur, and F.~Wiener  that $K_1=1/3$. In
\cite{BoasKhav97}, the two inequalities
\begin{equation}\label{BoasKhav}
\frac {1}{3} \sqrt{\frac{1}{n}} \le K_n \le 2 \sqrt{\frac{\log
n}{n}}
\end{equation}
were established for $n>1$. The paper of Boas and Khavinson
aroused new interest in the Bohr radius and has been a source of
inspiration for many subsequent papers. For some time (see for
instance \cite{Boas00}) it was thought that the left-hand side of
\eqref{BoasKhav} could not be improved. However, using
\eqref{old}, A.~Defant and L.~Frerick \cite{DefFrer06} showed that
\[
K_n\ge c \sqrt{\frac{\log n}{n\log\log n}}
\]
holds for some constant $c>0$.

Using Theorem~\ref{sidon}, we will prove the following estimate.
\begin{theorem}\label{bohrradius}
The $n$-dimensional Bohr radius $K_n$ satisfies
\[
K_n\ge \gamma \sqrt {\frac{\log n}{n}}
\]
for an absolute constant $\gamma>0$.
\end{theorem}

Combining this result with the right inequality in
\eqref{BoasKhav}, we conclude that
\[ K_n=b(n) \sqrt{\frac{\log n}{n}} \]
with $0<\gamma\le b(n) \le 2$. It is possible to extract from our
methods a numerical value for $\gamma$ larger than $0.2$, cf. the
concluding remark of Section~5.

\section{Preliminaries on multilinear forms}
The transformation of a homogeneous polynomial to a corresponding
multilinear form will play a crucial role in the proof of
Theorem~\ref{sidon}. We denote by $B$ an $m$-multilinear form in
$\C^n$, i.e., given $m$ points $z^{(1)},\ldots,z^{(m)}$ in $\C^n$,
we set
\[
B(z^{(1)},\ldots,z^{(m)})=\sum_\beta b_\beta
z_{\beta_1}^{(1)}\cdots z_{\beta_m}^{(m)}.
\]
We may express the coefficients as
$b_\beta=B(e^{\beta_1},\ldots,e^{\beta_m})$, where
$\{e^i\}_{i=1,\ldots,n}$ is the canonical base of $\C^n$. The form
$B$ is symmetric if for every permutation $\sigma$ of the set
$\{1,2,\ldots,m\}$,
$B(z^{(1)},\ldots,z^{(m)})=B(z^{(\sigma(1))},\ldots,z^{(\sigma(m))})$.
If we restrict a symmetric multilinear form to the diagonal
$P(z)=B(z,\ldots,z)$, then we obtain a homogeneous polynomial. The
converse is also true: Given a homogeneous polynomial $P:\C^n\to
\C$ of degree $m$, by polarization, we may define the symmetric
m-multilinear form $B: \C^{n}\times\cdots\times\C^n\to \C$ by
setting
\begin{equation}\label{polarization}
B(z^{(1)},\ldots,z^{(m)})=\frac {1}{2^m
m!}\sum_{\begin{subarray}{c}\varepsilon_i=\pm 1 \\ 1\le i\le
m\end{subarray}}\varepsilon_1\varepsilon_2\cdots\varepsilon_m
P\Bigl(\sum_{i=1}^m \varepsilon_i z^{(i)}\Bigr)
\end{equation}
so that $B(z,\ldots,z)=P(z)$. In what follows, $B$ will denote the
symmetric $m$-multilinear form obtained in this way from $P$.

We will consider the analogous norms for symmetric multilinear
forms as those introduced above. This means that we set
\[\| B \|_\infty=\sup_{\D^n\times\cdots\times \D^n} |B(z^{(1)},\ldots,z^{(m)})|
 \qquad \text{and} \qquad
 \lone B \rone=\sum_{|\beta|=m}
|b_\beta|.\]

It will be important for us to be able to relate the norms of $P$
and $B$. It is plain that $\|P\|_\infty \le \|B\|_\infty$. On the
other hand, it was proved by L.~Harris \cite{Harris75} that we
have, for non-negative integers $m_1,\cdots,m_k$ with
$m_1+\cdots+m_k=m$,
\begin{equation}\label{harris}
|B(\underbrace{z^{(1)},\ldots,z^{(1)}}_{m_1},\ldots,\underbrace{z^{(k)},\ldots,
z^{(k)}}_{m_k})|\le \frac{m_1!\cdots m_k!}{m_1^{m_1}\cdots
m_k^{m_k}}\frac{m^m}{m!}\
\|P\|_\infty;
\end{equation}
this result can be obtained from the polarization formula
\eqref{polarization}.

To compare the $\lone \cdot\rone$ norms, observe that the
coefficients $b_\beta$ of $B$  can be computed from the
corresponding coefficient of $P$: $b_\beta=c_\alpha/h(\beta)$,
where $h(\beta)$ is the number of different words that can be
formed with the letters in $\beta$. The corresponding $\alpha_j$
is the number of times any of the indices $\beta_i$ equals $j$. It
is therefore clear that
\[ \sum_{\alpha} |c_\alpha| = \sum_\beta |b_\beta|, \]
or, in other words, $\lone P \rone=\lone B\rone $.

\section{The tetrahedral part of a homogeneous polynomial}

A polynomial $Q(z)=\sum_\alpha c_\alpha z^\alpha$ is said to be
tetrahedral if $c_\alpha$ is nonzero only if $\max_j \alpha_j\le
1$; thus no term in $Q$ contains a factor of degree $2$ or higher
in any of the variables $z_1,\ldots,z_n$. Now set
$E=\{(\alpha_1,\ldots,\alpha_n): |\alpha|=m,\ \alpha_i\le 1,\
\forall i=1,\ldots,n\}.$ Then $T(P)=\sum_{\alpha\in E} c_\alpha
z^\alpha$ is the tetrahedral part of $P$ and $R(P)=P-T(P)$ is the
remainder corresponding to monomials containing a higher order
power in at least one of the variables $z_1,\ldots,z_n$.

In the next lemma, $p_1, p_2,\ldots$ are the prime numbers, listed
by increasing order, and $\sinc x=(\sin x)/x$.
\begin{lemma}\label{lemarep}
We have $\|T(P)\|_\infty \le \kappa^m \|P\|_\infty$ for every
homogeneous polynomial $P$ of degree $m$, where the constant
$\kappa$ can be taken as \[ \kappa=\left(\prod_{k=1}^\infty\
\sinc{\frac{\pi}{p_k}}\right)^{-1}=2.209\ldots .
\]
\end{lemma}
\begin{proof}
We will need the counting function for the prime numbers, which
will be denoted by $\varpi(x)$, in order not to confuse it with
the number $\pi$. We begin by constructing some auxiliary
functions. Set $Q=[0,1]^{\varpi(m)}$, let
$t=(t_1,\ldots,t_{\varpi(m)})$ denote a point in $Q$, and let
$d\mu$ be Lebesgue measure on $Q$. Define
\[
 r_m(t)=c_m \exp \left(2\pi i\Bigl(\frac {t_1}2 + \frac {t_2}3+\cdots+ \frac
{t_{\varpi(m)}}{p_{\varpi(m)}}\Bigr)\right),
\]
where
\[
 c_m=\prod_{k=1}^{\varpi(m)} \Bigl(\frac {p_k}{2\pi i}
\left(e^{\frac {2\pi i}{p_k}} -1\right)\Bigr)^{-1}.
\]
The functions $r_m:Q\to \C$ have the following properties:
\begin{enumerate}
 \item[(i)] $\int_Q r_m(t)\,d\mu(t)=1$,
 \item[(ii)] $\int_Q r_m^k(t)\, d\mu(t)=0$\ \ for all
 $k=2,\ldots,m$,
 \item[(iii)] $|r_m(t)|\le \kappa$\ \ for all $t$ in $Q$ and all $m>1$.
\end{enumerate}
It is immediate that (i) and (ii)
are satisfied. We note that (iii) also holds, because
$|r_m(t)|\equiv |c_m|$ and
\[
|c_m|^{-2}= \prod_{k=1}^{\varpi(m)} \frac{p_k^2}{(2\pi)^2}
\Bigl|e^{\frac{2\pi i}{p_k}}-1\Bigr|^2
 = \prod_{k=1}^{\varpi(m)} \sinc^2\frac{\pi}{p_k}.
\]
By properties (i) and (ii),
\[
T(P)(z)=\int_{Q^n} P(z_1r(t^1),\ldots, z_n r(t^n))\,
d\mu(t^1)\cdots d\mu(t^n),
\]
and so, by property (iii), $|P(z_1 r(t^1),\ldots, z_n
r(t^n))| \le \kappa^m \|P\|_\infty$ for every $z$ in $\D^n$.
\end{proof}
We can similarly define a decomposition of symmetric
$m$-multilinear forms. Let $F$ be the set of multiindices
$F=\{(\beta_1,\ldots,\beta_m):\ 1\le \beta_i \le n \text{ and all
indices }\beta_k \text{ are pairwise disjoint}\}$. Then we may
decompose $B=T(B)+R(B)$, where
\[
T(B)(z^{(1)},\ldots,z^{(m)})=\sum_{\beta\in F} b_\beta
z_{\beta_1}^{(1)}\cdots z_{\beta_m}^{(m)},
\]
Clearly, if $P$ is a homogeneous polynomial and $B$ its
corresponding symmetric multilinear form, then $T(P)$ has $T(B)$
as the corresponding multilinear form.
\section{Proof of Theorem~\ref{sidon}}

Since
\[ \lone P \rone = \lone R(P) \rone + \lone T(P) \rone, \]
it will suffice to obtain appropriate estimates for each of the
norms $\lone R(P) \rone$ and $\lone T(P) \rone$. The two lemmas
below together give the required bound for $\lone P \rone$.

We begin by estimating $\lone R(P) \rone $ in the range $n>m^2$.
\begin{lemma}\label{SC-S}
For a homogeneous polynomial $P$ of degree $m$ and $n>m^2>1$, we
have
\begin{equation}\label{coefrep}
\lone R(P)\rone \le \sqrt{2me \frac{n^{m-1}}{(m-1)!}}\
\|P\|_\infty.
\end{equation}
\end{lemma}
\begin{proof}
We begin by observing that the number of monomials $z^\alpha$ in $n$
variables of degree $m$ with $\max_j \alpha_j >1$ is
$\binom{n+m+1}{m}-\binom{n}{m}$. Thus the Cauchy--Schwarz inequality
gives \[ \lone R(P)\rone \le \sqrt{\binom{n+m+1}{m}-\binom{n}{m}}\
\| P\|_\infty. \] The result follows from this because
\[
\begin{split} \binom{n+m+1}{m}-\binom{n}{m}& \le
\frac{n^m}{m!}
\left[\left(1+\frac{m}{n}\right)^m-\left(1-\frac{m}{n}\right)^m\right]
\\ & \le \frac{n^m}{m!}\left[ e^{m^2/n} - e^{-m^2/n}\right] \le 2em
\frac{n^{m-1}}{(m-1)!}.
\end{split}
\]
\end{proof}

We turn next to the most challenging case, which is that of the
tetrahedral part $T(P)$. Now the Cauchy--Schwarz inequality does
not work because there are too many coefficients. We will transfer
the problem to $T(B)$ and use instead a special form of the
multilinear Khinchine inequality, which can be traced back to
\cite{Bonami70}. The precise formulation of the  result to be used
is in \cite[Theorem 3.2.2]{PenyaGine99}. In the theorem below,
$\{\epsilon_i\}_{i=1}^\infty$ denotes a Rademacher sequence of
random variables, i.e., the $\epsilon_i$ are i.i.d and $\mathbb
P(\epsilon_i=1)=\mathbb P(\epsilon_i=-1)=1/2$.
\begin{theorem}[Hypercontractivity] \label{hyper}
Let $X$ be a homogeneous chaos of order $m$: \[ X=\sum_{1\le
i_1<\cdots<i_m\le
n}x_{i_1,\ldots,i_m}\epsilon_{i_1}\cdots\epsilon_{i_m},
\]
with $x_{i_1,\ldots,i_m}\in \mathbb C$. Then
\[
\bigl(\mathbb E(|X|^2)\bigr)^{1/2} \le e^{m} \mathbb E(|X|).
\]
\end{theorem}

With this theorem we can prove the following.
\begin{lemma}\label{TSidon} For every homogeneous polynomial
$P$ of degree $m$ with $m<n$, we have
\[
\lone T(B)\rone\le (e\kappa)^{m} \binom{n-1}{m-1}^{1/2}
\|P\|_\infty,
\]
where $\kappa$ is the constant from Lemma~\ref{lemarep}.
\end{lemma}
\begin{proof} Put
\[
\begin{split}
F&=\{(i_1,\ldots, i_m), 1\le i_1,\ldots,i_m\le n \hbox { pairwise distinct}\},\\
F_{i_1}&=\{(i_2,\ldots, i_m) :\ (i_1,\ldots,i_m)\in F\},\\
\widetilde F_{i_1}&=\{(i_2,\ldots, i_m)\in F_{i_1},\ i_2<\cdots <i_m\}.
\end{split}
\]
We may write
\[
\begin{split}
\lone T(B)\rone
&=\sum_{(i_1,\ldots,i_m)\in F}|B(e^{i_1},\ldots,e^{i_m})|\\
& =\sum_{i_1=1}^n \sum_{(i_2,\ldots,i_m)\in F_{i_1}}
|B(e^{i_1},e^{i_2},\ldots,e^{i_m})|\\
&=\sum_{i_1=1}^n \sum_{(i_2,\ldots,i_m)\in \widetilde F_{i_1}}(m-1)!
|B(e^{i_1},e^{i_2},\ldots,e^{i_m})|\\
&\le \binom {n-1}{m-1}^{1/2} \sum_{i_1=1}^n  \
\Bigl(\!\!\!\sum_{(i_2,\ldots,i_m)\in \widetilde
F_{i_1}}\hspace{-5mm}
((m-1)!|B(e^{i_1},e^{i_2},\ldots,e^{i_m})|)^2\Bigr)^{1/2}.
\end{split}
\]
By Theorem~\ref{hyper},
\[
\begin{split}
&\Bigl(\sum_{(i_2,\ldots,i_m)\in \widetilde F_{i_1}}((m-1)!|
B(e^{i_1},e^{i_2},\ldots,e^{i_m})|)^2\Bigr)^{1/2}\\
&\qquad \le e^{m-1}\,\mathbb E\, \Bigl(\bigl|
\sum_{(i_2,\ldots,i_m)\in \widetilde
F_{i_1}}(m-1)!B(e^{i_1},e^{i_2},\ldots,e^{i_m})
\epsilon_{i_2}\cdots \epsilon_{i_m}\bigr|\Bigr).
\end{split}
\]
Summing over $i_1$, we get
\begin{equation} \label{First}
\lone T(B)\rone  \le e^{m-1}\binom {n-1}{m-1}^{1/2}\!\! \sup_{z\in
\D^n}\sum_{i_1=1}^n\ \Bigl|\!\! \sum_{(i_2,\ldots,i_m)\in
F_{i_1}}\hspace{-4mm}
B(e^{i_1},e^{i_2},\ldots,e^{i_m})z_{i_2}\cdots z_{i_m}\Bigr|.
\end{equation}
We introduce the notation \[ \lambda_{i_1}(z)=\frac{\bigl|
\sum_{(i_2,\ldots,i_m)\in
F_{i_1}}B(e^{i_1},\ldots,e^{i_m})z_{i_2}\cdots z_{i_m}\bigr|}
{\sum_{(i_2,\ldots,i_m)\in
F_{i_1}}B(e^{i_1},\ldots,e^{i_m})z_{i_2}\cdots z_{i_m}}\] and
obtain from \eqref{First}
\[
\begin{split}
\lone T(B) \rone &\le e^{m-1}\binom {n-1}{m-1}^{1/2}\!\! \sup_{z\in
\D^n}\!\! \sum_{(i_1,\ldots,i_m)\in F}\hspace{-5mm}
B(e^{i_1},e^{i_2},\ldots,e^{i_m})\lambda_{i_1}(z)z_{i_2}\cdots
z_{i_m}\\
& \le e^{m-1}\binom {n-1}{m-1}^{1/2} \sup_{(z^{(1)},z^{(2)})\in
\D^n\times \D^n}|T(B)(z^{(1)},\underbrace{z^{(2)},\ldots,z^{(2)}}_{m-1})|\\
&\le e^{m} \binom
{n-1}{m-1}^{1/2}\|T(P)\|_\infty \\
& \le e^{m} \kappa^m \binom {n-1}{m-1}^{1/2}\|P\|_\infty,
\end{split}
\]
where in the last two steps we used Harris's inequality
\eqref{harris} and Lemma~\ref{lemarep}.
\end{proof}

\section{Proof of Theorem~\ref{bohrradius}}

For the proof of Theorem~\ref{bohrradius}, we need the following
lemma of F. Wiener (see \cite{BoasKhav97}).
\begin{lemma}\label{FWiener}
Let $P$ be a polynomial in $n$ variables and $P=\sum_{m\ge 0} P_m$
its expansion in homogeneous polynomials. If $\|P\|_\infty \le 1$,
then $\|P_m\|_\infty \le 1-|P_0|^2$ for every $m>0$.
\end{lemma}

 \begin{proof}[Proof of Theorem~\ref{bohrradius}] We assume that $\sup_{\D^n}
\bigl|\sum c_\alpha z^\alpha \bigr|\le 1$. Observe that for all $z$ in
$r\Dn$,
\[
\sum |c_\alpha z^\alpha|\le |c_0|+\sum_{m>1} r^m\sum_{|\alpha|=m}|c_\alpha|.
\]
When $m>\log n$, we use \eqref{trivial} and Lemma~\ref{FWiener},
and obtain the estimate
\begin{equation}\label{bigm}
 \sum_{m>\log n} r^m\sum_{|\alpha|=m}|c_\alpha| \le  \sum_{m>\log n}
r^m \sqrt{\binom{n+m-1}{m}} (1-|c_0|^2),
\end{equation}
whence
\begin{equation}\label{mbig}
 \sum_{m>\log n} r^m\sum_{|\alpha|=m}|c_\alpha| \le  \sum_{m>\log n}
r^m (2e)^m \max(1,n/m)^{m/2} (1-|c_0|^2).
\end{equation}
If we take into account the estimate
\[
\frac{(\log n)^{m}}{n}\le m!
\]
(obtained by a calculus argument), then Theorem~\ref{sidon} and
Lemma~\ref{FWiener} give
\begin{equation}\label{msmall}
 \sum_{m<\log n} r^m\sum_{|\alpha|=m}|c_\alpha| \le  \sum_{m<\log n}
r^m \sqrt{m}\ C^m\Bigl(\frac{n}{\log n}\Bigr)^{m/2} (1-|c_0|^2).
\end{equation}
If we now choose $r\le \varepsilon \sqrt{\frac{\log n}{n}}$ with
$\varepsilon$ small enough and combine \eqref{mbig} and
\eqref{msmall}, we obtain
\[
\sum |c_\alpha z^\alpha|\le|c_0| + (1-|c_0|^2)/2\le 1
\]
whenever $|c_0|\le 1$. Thus the theorem is proved with
$\gamma=\varepsilon$ .\end{proof}

A closer examination of this proof shows that a better choice
would be to use Theorem~\ref{sidon} only when $m<(2+2\log
\kappa)^{-1}\log n$. By this approach and taking into account the
estimates from Lemmas~\ref{SC-S} and \ref{TSidon}, we get
\[
b(n)\ge \frac{1}{\sqrt{2e(1+\log \kappa)}}+o(1)
\]
when $n\to \infty.$ By also doing a meticulous analysis of
\eqref{bigm} for ``small'' $n$ and keeping in mind that $S(1,n)=1$,
one may arrive at a numerical value for $\gamma$ which is larger
than $0.2$.


\begin{thebibliography}{dlPG99}


\bibitem%[Boa]
{Boas00}
H.~P. Boas, \emph{Majorant series}, J. Korean Math. Soc. \textbf{37}
  (2000), 321--337.

\bibitem%[BK]
{BoasKhav97}
H.~P. Boas and D. Khavinson, \emph{Bohr's power series theorem in
  several variables}, Proc. Amer. Math. Soc. \textbf{125} (1997),
  2975--2979.

\bibitem%[BH]
{BH31}H.~F.~Bohnenblust and E. Hille, \emph{On the absolute
convergence of Dirichlet series}, Ann. of Math. (2) \textbf{32}
(1931), 600--622.

\bibitem%[Boh]
{Bohr14}
H. Bohr, \emph{A theorem concerning power series}, Proc. London
Math. Soc.  \textbf{13} (1914), 1--5.

\bibitem%[Bon]
{Bonami70}
A. Bonami, \emph{\'{E}tude des coefficients de {F}ourier des
fonctions de
  {$L\sp{p}(G)$}}, Ann. Inst. Fourier (Grenoble) \textbf{20} (1970),
  335--402.

\bibitem%[DDGM]
{DDGM01} A. Defant, J.~C. D\'{\i}az. D. Garc\'{\i}a, and M.
Maestre, \emph{Unconditional basis and Gordon--Lewis constants for
spaces of polynomials}, J. Funct. Anal. \textbf{181} (2001),
119--145.


\bibitem%[DF]
{DefFrer06}
A. Defant and L. Frerick, \emph{A logarithmic lower bound for
  multi-dimensional {B}ohr radii}, Israel J. Math. \textbf{152} (2006), 17--28.

\bibitem%[dlPG]
{PenyaGine99}
V.~H. de~la Pe{\~n}a and E. Gin{\'e}, \emph{Decoupling},
  Probability and Its Applications (New York), Springer-Verlag, New York, 1999.

\bibitem%[Har]
{Harris75}
L.~A. Harris, \emph{Bounds on the derivatives of holomorphic
functions of
  vectors}, Analyse fonctionnelle et applications ({C}omptes {R}endus {C}olloq.
  {A}nalyse, {I}nst. {M}at., {U}niv. {F}ederal {R}io de {J}aneiro, {R}io de
  {J}aneiro, 1972), Hermann, Paris, 1975, pp.~145--163. Actualit\'es Aci.
  Indust., No. 1367.

\bibitem%[Kah]
{Kahane}
J.-P. Kahane, \emph{Some Random Series of Functions}, Second
edition, Cambridge Studies in Advanced Mathematics, 5. Cambridge
University Press, Cambridge, 1985.

\bibitem%[Que]
{Que} H. Queff\'{e}lec, \emph{Harald Bohr's
vision of Dirichlet series; old and new results}, J. Anal.
\textbf{3} (1995), 43--60.

\end{thebibliography}
\end{document}